\documentclass[english,10pt]{amsart}
\usepackage{ucs}
\usepackage[utf8x]{inputenc}
\usepackage{babel}
\usepackage{amsmath,amsthm}
\usepackage{amssymb}
\usepackage[all]{xy}
\usepackage{verbatim}

\newcommand{\R}{\mathbb{R}}
\newcommand{\C}{\mathbb{C}}
\newcommand{\Q}{\mathbb{Q}}
\newcommand{\Z}{\mathbb{Z}}

\newcommand{\ra}{\rightarrow}

\newtheorem{thm}{Theorem}
\newtheorem{df}[thm]{Definition}
\newtheorem{prop}[thm]{Proposition}
\newtheorem{lem}[thm]{Lemma}
\newtheorem{cor}[thm]{Corollary}
\newtheorem{conj}[thm]{Conjecture}

\newtheorem{question}{Question}

\theoremstyle{remark}
\newtheorem{rk}[thm]{Remark}

\title{On the Picard number of $K3$ surfaces over number fields} 
\author{Fran\c{c}ois Charles}
\email{francois.charles@univ-rennes1.fr}
\address{Université de Rennes 1, IRMAR--UMR 6625 du CNRS, Campus de Beaulieu, 35042 Rennes
Cedex, France}
\thanks{This paper was written during a stay at the university of Bonn. I would like to thank Daniel Huybrechts for his hospitality and many useful comments.}

\begin{document}

\begin{abstract}
 We discuss some aspects of the behavior of specialization at a finite place of
N\'eron-Severi groups of $K3$ surfaces over number fields. We give optimal lower bounds for
the Picard number of such specializations, thus answering a question of Elsenhans and Jahnel.
As a consequence of these results, we show that it is possible to explicitly compute the
Picard number of any given $K3$ surface over a number field.
\end{abstract}

\maketitle

\section{Introduction}

This paper deals with two questions concerning the arithmetic and the geometry
of $K3$ surfaces. Let $X$ be a polarized $K3$ surface over a number field $k$, and let
$\mathfrak p$ be a finite place of $k$ where $X$ has good reduction. Denote by $X_{\mathfrak
p}$ the special fiber of a smooth model of $X$ over the ring of integers of $k_{\mathfrak p}$.
Denote by $\overline{X}$ (resp. $\overline{X_{\mathfrak p}}$) the base change of $X$ (resp.
$X_{\mathfrak p}$) to an algebraic closure of $k$ (resp. the residue field of $\mathfrak p$). 
Specialization of divisors induces a specialization map between the N\'eron-Severi groups of
$\overline{X}$ and $\overline{X_{\mathfrak p}}$.

\begin{question}\label{sp}
 What can be said about the specialization map
\begin{equation}\label{speq}
 sp : NS(\overline{X})\ra NS(\overline{X_{\mathfrak p}})\,?
\end{equation}
\end{question}

A standard argument using the cycle class map and the smooth base change theorem shows that
this specialization map is always injective. We are here interested in the defect of
surjectivity.

The second question is the following. Recall that the Picard number of a variety is by
definition the rank of its N\'eron-Severi group.

\begin{question}\label{rho}
Given a projective embedding of $X$, is it possible to compute the Picard number of $X$~?
\end{question}
This question is raised by Shioda in \cite{Sh81}.

Using the Weil conjectures \cite{De74}, it is possible to compute the Picard numbers
of smooth projective varieties over finite fields. Indeed, counting points in sufficiently many
extensions of the base field, one can compute the characteristic polynomial of the Frobenius
acting on the second \'etale cohomology group, and determine the multiplicity of $1$ as an
eigenvalue. If the Tate conjecture holds, this multiplicity is equal to the
Picard number.

\bigskip

In characteristic zero, Question \ref{rho} is more difficult. In particular, the first explicit
example of a K3 surface over a number field with Picard rank $1$ has been recently given by van
Luijk in \cite{vL07}. Van Luijk's method provides a link between both questions. Indeed, it
proceeds by computing Picard numbers at sufficiently many finite places in order to get
information over the field of definition. In the past few years, the problem of computing Picard numbers of $K3$ surfaces has been featured for instance in the work of Elsenhans-Jahnel \cite{EJ08}, \cite{EJ082}, with recent geometric applications in the work of Hassett, V{\'a}rilly-Alvarado and V\'arilly, \cite{HVAV11} and \cite{HVA11}.

With this approach, one of the main problems is finding finite places $\mathfrak p$ such that
the specialization map (\ref{speq}) is as close to being surjective as possible, i.e., such
that $\rho(\overline{X_{\mathfrak p}})$ is as small as possible.

\bigskip

Note that the situation in this mixed characteristic setting is in stark contrast with the
case of equal characteristic zero. Indeed, for $K3$ surfaces defined over function fields over
$\C$
or $\overline\Q$, most specializations induce isomorphisms at the level of the N\'eron-Severi
group. This is a consequence of Baire's theorem over $\C$, see for instance \cite{Vo02},
Chapter 13, and of the Hilbert irreducibility theorem over $\overline \Q$, see \cite{Te85} and
\cite{El02}. A different approach to this problem can be found in \cite{MP09}.

On the other hand, over finite fields, there are obstructions for the map (\ref{speq}) to be
surjective, as was first noticed by Shioda in \cite{Sh81} and \cite{Sh83}. Indeed, it is a
consequence of the Tate conjecture that the geometric Picard number of a $K3$ surface over a
number field is always even, see for instance \cite{dJK00}. This striking fact has been
recently used in a surprising way by Bogomolov-Hassett-Tschinkel in \cite{BHT11} and Li-Liedtke
in \cite{LL10} to prove that any complex $K3$ surface with odd Picard rank contains infinitely
many rational curves.

In this paper, we describe the Shioda-type obstructions that can prevent the map (\ref{speq})
from being surjective, and we give optimal lower bounds for the Picard number of the
specialization. One of our results is that Hodge theory can force the existence of such obstructions even when the Picard number is even, see part (2) of Theorem \ref{jump} below. 
\bigskip

Let $X$ be a $K3$ surface over a number field $k$, and choose a complex embedding of $k$. Let
$\rho$ be the geometric Picard number of $X$ and, for any finite place $\mathfrak p$ of $k$
where $X$ has good reduction, let $\rho_{\mathfrak p}$ be the geometric Picard number of
$X_{\mathfrak p}$. Note that we always have $$\rho_{\mathfrak p}\geq \rho.$$ 

We need to control the Hodge theory of $X_{\C}$. Let $T$ be the orthogonal of $NS(X_{\C})$ in
the singular cohomology group $H^2(X_{\C}, \Q)$ with respect to cup-product. The space $T$ is a
sub-Hodge structure of $H^2(X_{\C}, \Q)$. Let $E$ be the algebra of endomorphisms of $T$ that
respect the Hodge structure. In \cite{Za83}, Zarhin shows that $E$ is either a
totally real field or a CM field.

The following result can be considered as a number field
analog of the specialization results over function fields mentioned above.

\begin{thm}\label{jump}
Let $X$, $T$ and $E$ be as above.
\begin{enumerate} 
 \item If $E$ is a CM field or the dimension of $T$ as an $E$-vector space is even, then there
exist infinitely many places $\mathfrak p$ of good
reduction such that $\rho_{p}=\rho$. Furthermore, after replacing $k$ by a finite
extension, this equality holds for a set of places of density $1$.
     
     \newpage
     
 \item Assume $E$ is a totally
real field and the dimension of $T$ as an $E$-vector space is odd. 

Let $\mathfrak p$ be a finite place of $k$ where $X$ has good reduction. If $X_{\mathfrak p}$
satisfies the Tate conjecture, then
$$\rho_{\mathfrak p}\geq \rho+[E:\Q].$$

There exist infinitely many places
$\mathfrak p$ of good reduction such that $\rho_{\mathfrak p}=\rho+[E:\Q]$. Furthermore, after
replacing $k$ by a finite extension, this equality holds for a set of places of density
$1$.
\end{enumerate}
\end{thm}

\begin{rk}
 Note that if $\rho$ is odd, $X$ satisfies the assumptions of the second part of the
theorem.
\end{rk}

\begin{rk}
 By work of Nygaard and Nygaard-Ogus in \cite{Ny83}, \cite{NO85}, the Tate conjecture holds for
ordinary $K3$ surfaces
over finite fields and non-supersingular $K3$ surfaces over fields of characteristic at least
$5$.
\end{rk}

\begin{rk}\label{cex}
 In \cite{EJ11}, Elsenhans and Jahnel ask whether, with notations as in the theorem, there
exists $\mathfrak p$ such that $\rho_{\mathfrak p}-\rho\leq 1$. The result above shows that it
is not the case if $E$ is a totally real field of degree at least $2$ over $\Q$, such that the
dimension of $T$ over $E$ is odd. This is however true in all other cases.
\end{rk}

 This result shows that the Picard number can be forced to jump in specializations even when
the Picard number of $X$ is even. Using the method of Li and Liedtke in \cite{LL10}, we get the following corollary.

\begin{cor}
Let $X$ be either a $K3$ surface of Picard rank $2$ with $E$ a totally real field of degree $4$ or a $K3$ surface of Picard
rank $4$ with $E$ a totally real field of even degree. Then $X$ contains infinitely many rational curves.
\end{cor}

There exist such $K3$ surfaces by \cite{vG08}, section 3, and they give new examples of $K3$ surfaces with infinitely many rational curves. Note that complex $K3$ surfaces of Picard rank different from $2$ and $4$ are known to contain infinitely many rational curves by \cite{LL10}.

\bigskip

The second main result of this paper is a solution to Question \ref{rho}. Recall that van
Luijk's method in \cite{vL07} to prove that a $K3$ surface $X$ over $\Q$ has Picard number $1$
was to first find two primes $p$ and $q$ of good reduction such that $X$ specializes to a $K3$
surface
of Picard number $2$ modulo $p$ and $q$. If the discriminant of the N\'eron-Severi lattices
modulo $p$ and $q$ differ by a non-square factor, van Luijk shows that this implies that $X$
has Picard number $1$.

By Remark \ref{cex}, there are cases where we cannot expect van Luijk's method to work directly
for all $K3$ surfaces of rank $1$. However, the second part of Theorem \ref{jump} can be used
to show that reduction at finite places does indeed give enough information to compute Picard
numbers over number fields. 

This gives a theoretical explanation to the computations in \cite{vL07}, \cite{EJ08}, \cite{EJ082}, \cite{HVAV11}, \cite{HVA11}.

\begin{thm}\label{compute}
 There exists an algorithm which, given a projective $K3$ surface $X$ over a number field,
either returns its geometric Picard number or does not terminate. 

If $X\times X$ satisfies the Hodge conjecture for codimension $2$ cycles, then the algorithm applied to $X$ terminates.
\end{thm}

\begin{rk}
Let $X$ be a $K3$ surface over $\C$. With the notations of Theorem \ref{jump}, $X\times X$
satisfies the Hodge conjecture if and only if the field $E$ acts by algebraic correspondences.
By \cite{A96}, this would be a consequence of the standard conjectures. In \cite{Mu02}, Mukai
has announced a proof in the case $E$ is a CM field.
\end{rk}

\begin{rk}
 The proof of the theorem actually shows that the only case where the algorithm would not
terminate is, with the notations of Theorem \ref{jump}, if $E$ is a totally real field that
does not act on $H^2(X, \Q)$ by algebraic correspondences and $T$ is of odd dimension as a
vector space over $E$.

In particular, the algorithm always terminates for surfaces with $E=\Q$.
\end{rk}

\bigskip

While we only consider $K3$ surfaces in this paper, some of the methods we consider have a
wider range of applications. Assuming general conjectures on algebraic cycles, it is a general
fact that the Mumford-Tate group associated to the second cohomology group of a variety
controls specialization of N\'eron-Severi groups, in a fashion that is similar to the way the
monodromy representation appears in \cite{El02} or \cite{MP09}. The multiplicity of the weight zero in the corresponding representation is what forces the Picard number to jump after specialization. This is related to algorithmic computations of N\'eron-Severi groups as in our paper. 

For $K3$ surfaces, the work of Zarhin and Tankeev in \cite{Za83} and \cite{Ta90}, \cite{Ta95}
allows us to give precise and unconditional results. The results of our paper conjecturally
hold for varieties with $h^{2, 0}=1$. It seems likely that one can prove them unconditionally
for holomorphic symplectic varieties by extending the work of Tankeev cited above.

\bigskip

In section 2, we recall results of Zarhin and Tankeev on the second cohomology group of a $K3$
surface. This allows us to prove Theorem \ref{jump} in section 3. Section 4 is devoted to
discriminant computations which will allow us to prove Theorem \ref{compute} in the last
section.

\section{Algebraic monodromy groups of $K3$ surfaces over number fields}

The results of this section are mostly contained in the work of Zarhin and Tankeev. After
recalling some
preliminary material, we describe the algebraic monodromy group of a $K3$ surface defined over
a number field.

\subsection{Mumford-Tate groups and the Mumford-Tate conjecture}

Let $\mathbb S$ be the Deligne torus, that is, the algebraic group over $\R$ defined as
$$\mathbb S=\mathrm{Res}_{\C/\R} \mathbb G_m.$$
Let $H$ be a finite-dimensional vector space over $\Q$. Giving a Hodge structure on $H$ is
equivalent to giving an action of $\mathbb S$ on $H_{\R}=H\otimes\R$.

\begin{df}
 Let $H$ be a rational Hodge structure. The Mumford-Tate group of $H$ is the smallest
algebraic subgroup $MT(H)$ of $GL(H)$ such that $MT(H)_{\R}$ contains the image of
$\mathbb S$ in $GL(H_{\R})$.
\end{df}

We refer to \cite{DMOS82}, Chapter I, for general properties of Mumford-Tate groups. Since
$\mathbb S$ is connected, this definition implies that Mumford-Tate groups are connected.
Note that the Mumford-Tate group of a polarized Hodge structure is reductive.

Let $i,j$ be nonnegative integers, and consider the Hodge structure 
$$V=H^{\otimes i}\otimes (H^*)^{\otimes j}.$$
The Mumford-Tate group $MT(H)$ acts on $V$. If $v$ is a Hodge class in $V$, then the line $\Q
v$ is globally invariant under the action of $MT(H)$. Conversely, it follows from Chevalley's
theorem on affine groups that $MT(H)$ is the largest algebraic subgroup of
$GL(H_{\C})$ that
leaves all such lines globally invariant, see \cite{DMOS82}.

\bigskip

We now turn to the $\ell$-adic theory. General results can be found in \cite{Se81}. Let $k$ be a
number field and fix an algebraic closure $\overline k$. Let $X$ be a smooth projective variety
over $k$, and denote by $\overline{X}$ the variety $X\times_{\mathrm{Spec} k}
\mathrm{Spec}\,\overline k$. Fixing a prime number $\ell$, we can consider the \'etale
cohomology
group $H^i(\overline X, \Q_{\ell})$ for some integer $i$.
Let $\rho_{\ell}$ denote the continuous representation 
$$\rho_{\ell} : G_k\ra GL(H^i(\overline X, \Q_{\ell}))$$
of the absolute Galois group $G_k$ of $k$. The image of $\rho_{\ell}$ is an $\ell$-adic Lie
group.

\begin{df}
 With notations as above, let $G_{\ell}$ be the Zariski closure of the image of
$\rho_{\ell}$ in the algebraic group $GL(H^i(\overline X, \Q_{\ell}))$. The algebraic group
$G_{\ell}$ is
called the
algebraic monodromy group associated to the Galois representation $\rho_{\ell}$.
\end{df}

Note that replacing $k$ by a finite extension replaces $G_{\ell}$ by an open subgroup of finite
index. In particular, the neutral component of the algebraic monodromy group does not depend on
the choice of a field of definition for $X$.

\bigskip

General conjectures on algebraic cycles give important information on Mumford-Tate and
algebraic monodromy groups. In particular, the latter are expected to be reductive. The
expected relationship between those two groups is described by the Mumford-Tate conjecture as
follows, see \cite{Se81}.

\begin{conj}
 Let $k$ be a number field and fix a complex embedding of $k$. Let $X$ be a smooth projective
variety over $k$.

Let $G_{\ell}$ be the algebraic monodromy group associated to the \'etale cohomology group
$H^i(X_{\C}, \Q_{\ell})$ for some prime number $\ell$, and let $G_{\ell}^{\circ}$ be
its neutral component. Then there is a canonical isomorphism
$$G_{\ell}^{\circ}\simeq MT(H^i(X_{\C}, \Q))_{\Q_{\ell}}.$$
\end{conj}

The Mumford-Tate conjecture is implied by the conjunction of the Tate and Hodge conjectures. A
lot of work has been done in its direction in the case of abelian varieties, see for instance
\cite{Se68}, \cite{Pi98}, \cite{Va08}. 

In this paper, we
will focus on the case of K3 surfaces, where the Mumford-Tate conjecture holds. However, an
important part of our method concerning specialization of N\'eron-Severi groups holds in a
general setting if one assumes the Mumford-Tate conjecture.

\subsection{Mumford-Tate groups and algebraic monodromy groups of K3 surfaces}

The following result is due to Tankeev and is crucial to this
paper.

\begin{thm}[Tankeev, \cite{Ta90}, \cite{Ta95}]\label{MT}
 The Mumford-Tate conjecture holds for the second cohomology group of $K3$ surfaces over number
fields.
\end{thm}

This result allows for a Hodge-theoretic description of the Galois action on the second
cohomology group of a $K3$ surface.

\bigskip

Let us now recall the description due to Zarhin in \cite{Za83} of the Mumford-Tate group of a
$K3$ surface.
Let $X$ be a $K3$ surface over $\C$, and consider the singular cohomology $H=H^2(X, \Q)$
endowed with its weight $2$ Hodge structure. The Hodge structure $H$ splits as a direct sum
$$H=NS(X)\oplus T,$$
where $NS(X)$ is the N\'eron-Severi group of $X$ with rational coefficients, and $T$ is the
orthogonal of $NS(X)$ in $H$ with respect to the cup-product. The Hodge structure $T$ is
called the transcendental part of $H^2(X, \Q)$.

The Hodge structure $T$ is
simple. By Lefschetz's theorem on $(1,1)$ classes, $T$ is the smallest sub-Hodge structure of
$H$ such that $T\otimes\C$ contains $H^2(X, \mathcal O_X)$. By the Hodge index theorem,
cup-product on $H^2(X, \Q)$ restricts to a polarization $\psi: T\otimes T\ra \Q$ on $T$.

Since $NS(X)$ is spanned by Hodge classes, the Mumford-Tate group of $H$ acts by a character on
$NS(X)$ and identifies with the Mumford-Tate group of $T$. Since $T$ is polarized by $\psi$,
$MT(T)$ is contained in the group of orthogonal similitudes $GO(T, \psi)$.

\bigskip

Let $E$ be the algebra of endomorphisms of the Hodge structure $T$. In \cite{Za83}, Zarhin
proves
that $E$ is either a totally real field or a CM field. The field $E$ is equipped with an
involution induced by the polarization on $T$, which is either the identity if $E$ is totally
real or complex conjugation in case $E$ is CM.

Since $E$ consists of endomorphisms of Hodge structures, the Mumford-Tate group of $T$
commutes with $E$. By the discussion above, the Mumford-Tate group of $T$ is a subgroup of the
centralizer of $E$ in the group $GO(T, \psi)$.

\begin{thm}[Zarhin, \cite{Za83}]\label{MTK3}
The Mumford-Tate group of $T$ is the centralizer of $E$ in the group of orthogonal similitudes
$GO(T, \psi)$.
\end{thm}

\bigskip

Now keep the same notation, and assume $X$ can be defined over a number field $k$. Fix a prime
number $\ell$. The action of the absolute Galois group $G_k$ on $H^2(X, \Q_{\ell})$ leaves the
$\Q_{\ell}$-span of the N\'eron-Severi group of $X$ globally invariant, as well as its
orthogonal
$T_{\ell}=T\otimes \Q_{\ell}$. As above, the neutral component of the algebraic monodromy group
$G_{\ell}$
of $H^2(X, \Q_{\ell})$ identifies with the algebraic monodromy group of $T_{\ell}$.

The polarization $\psi$ on $T$ extends to a symmetric bilinear form $\psi_{\ell}$. The
representation of $G_k$ in the automorphism group of $T_{\ell}$ factors through the
group $GO(T_{\ell}, \psi_{\ell})$.

Since Hodge cycles on products of $K3$ surfaces are absolute Hodge, see \cite{DMOS82}, the
field $E$
corresponding to endomorphisms of the Hodge structure $T$ acts on $T_{\ell}$ and commutes with
a finite-index subgroup of $G_k$. As a consequence, the neutral component of $G_{\ell}$
commutes with the action of $E\otimes\Q_{\ell}$.

By Theorem \ref{MT}, the Mumford-Tate conjecture holds for $X$. As an immediate corollary of
Theorem \ref{MTK3}, we get the following description of the neutral component of the algebraic
monodromy group of $X$.

\begin{cor}\label{algK3}
 With notations as above, the neutral component of the algebraic monodromy group associated to
$T_{\ell}$ is the centralizer of $E\otimes \Q_{\ell}$ in the group of orthogonal similitudes
$GO(T_{\ell},
\psi_{\ell})$.
\end{cor}

\section{Picard numbers of specializations}

This section is devoted to the proof of Theorem \ref{jump}. We start by the following result
which encompasses the elementary linear algebra needed in Theorem \ref{jump}.

Let $T$ be a finite dimensional vector space endowed with a symmetric bilinear form $\psi$. If
$f$ is any linear endomorphism of $T$, let $f'$ be the adjoint of $f$ with respect to
$\psi$. 

Let $E$ be a number field acting on $T$. Assume that $E$ is stable under $e\mapsto e'$, and
that $E$ is either a totally real field with $e=e'$ for all $e\in E$, or a CM field such that
$e\mapsto e'$ acts as complex conjugation on $E$.

Let $H$ be the centralizer of $E$ in the special orthogonal group $SO(T, \psi)$. Let
$\ell$ be a prime number, and let $H_{\ell}=H\otimes\Q_{\ell}$. 
\begin{prop}\label{eigenvalues}
The following holds.
\begin{enumerate}
 \item If $E$ is a CM field or the dimension of $T$ as an $E$-vector space is even, then there
exists $h\in H_{\ell}$ such that $h$ does not have any root of unity as an
eigenvalue.

  \item If $E$ is a totally real field and the dimension of $T$ as an $E$-vector space is
odd, then the eigenspace of any $h\in H_{\ell}$ associated to the eigenvalue $1$ is
of dimension at least $[E:\Q]$. Furthermore, there exists $h\in H_{\ell}$ for which
this dimension is exactly $[E:\Q]$ and such that no root of unity different from $1$ appears as
an eigenvalue of $h$.
\end{enumerate}
\end{prop}

\begin{proof}

Let us first assume that $E$ is a totally real field. By \cite{Za83}, 2.1, there
exists a unique $E$-bilinear form $\phi : T\times T\ra E$ such that $\psi=Tr_{E/\Q}(\phi)$.
With this notation,the centralizer of $E$ in $SO(T, \psi)$ is equal, as a subgroup of $GL(T)$,
to the Weil restriction $Res_{E/\Q}(SO_E(T, \phi))$, where $SO_E(T, \phi)$ denotes the group of
orthogonal similitudes of the $E$-vector space $T$ with respect to $\phi$.

\bigskip
Assume furthermore that the dimension of $T$ as a vector space over $E$ is even, and let us
show that there is an element $h\in H_{\ell}$ such that $g$ does not have any root
of unity as an eigenvalue. 

Considering an orthogonal decomposition of $T$ as an $E$-vector
space endowed with the bilinear form $\phi$, we can assume $T$ is of dimension $2$ over $E$.
Let $h$ be an orthogonal automorphism of the $E$-vector space $T$ of determinant $1$ that is
not of finite order. Then $h$ corresponds to an element of $H_{\ell}$ with the
desired property.

\bigskip

Now if the dimension of $T$ as a vector space over $E$ is odd, recall that any
element of $SO_E(T, \psi)$ admits $1$ as an eigenvalue. It follows from the description
of $H_{\ell}$ as a Weil restriction that any $h\in H_{\ell}$ has $1$ as an eigenvalue,
and that the corresponding eigenspace is invariant under the action of $E$. As a consequence,
its dimension is at least $[E:\Q]$. One can then argue as in the previous paragraph to
conclude the proof of the theorem in this case.

\bigskip

Let us now assume that $E$ is a CM field. Let $e$ be an element of $E$ such that $ee'=1$ and
$e$ is not a root of unity. Then multiplication by $e$ on $T$ corresponds to an element of
$H_{\ell}$ as in the theorem.
\end{proof}

We now turn to the proof of Theorem \ref{jump}. From now on, we use the notations there. Let
us start with a straightforward lemma. 

\begin{lem}\label{SO}
 The neutral component of the algebraic monodromy group associated to
$T_{\ell}(1)$ is the centralizer of $E\otimes \Q_{\ell}$ in the special orthogonal group
$SO(T_{\ell},
\psi_{\ell})$.
\end{lem}

\begin{proof}
 The representation of $G_k$ on $T_{\ell}(1)$ is equal to the representation of $G_k$ on
$T_{\ell}$
twisted by the cyclotomic character. On the other hands, general properties of \'etale
cohomology show that $G_k$ acts on $T_{\ell}(1)$ through the orthogonal group $O(T_{\ell},
\psi_{\ell})$. 

The lemma then follows from Corollary \ref{algK3} and the fact that the special
orthogonal group is the neutral component of the orthogonal group.
\end{proof}

\begin{proof}[Proof of Theorem \ref{jump}]
We use the notations of the theorem. First note that since specialization of N\'eron-S\'everi
groups is injective, the inequality $\rho_{\mathfrak p}\geq \rho$ always holds.

Let $F_{\mathfrak p}$ be the geometric Frobenius at $\mathfrak p$ acting on the \'etale
cohomology group $H^2(\overline{X_{\mathfrak p}}, \Q_{\ell}(1))$, where $\ell$ is a prime
number prime
to $\mathfrak p$. By the smooth base change theorem, the group $H^2(\overline{X_{\mathfrak p}},
\Q_{\ell}(1))$ identifies with $H^2(\overline{X}, \Q_{\ell}(1))$, and $F_{\mathfrak p}$ leaves
both the N\'eron-Severi group and $T_{\ell}(1)$ globally invariant. 

Let $H$ be the centralizer of $E\otimes \Q_{\ell}$ in the special
orthogonal group $SO(T_{\ell}, \psi_{\ell})$. Let $n$ be the dimension of $T$ as a vector space
over $\Q$, and let $S$ be the finite set of complex roots of unity of degree at most $n$ over
$\Q$.

\bigskip

Assume first that $E$ is a CM field or $E$ is a totally real field and the dimension of $T$ as
a vector space over $E$ is even.  By Proposition \ref{eigenvalues}, the set of $h\in H_{\ell}$
such that $h$ does not have any eigenvalue in $S$ is a dense, Zariski-open subset of
$H_{\ell}$.

By Lemma \ref{SO} and Chebotarev's density theorem, we can find a finite extension $k'$ of
$k$ and a set $U$ of finite places $\mathfrak p$ of $k'$ that has density $1$ such
that for any $\mathfrak p\in U$, $X$ has good reduction at $\mathfrak p$ and the geometric
Frobenius $F_{\mathfrak p}$ acting on $T_{\ell}(1)$ does not have any eigenvalue in $S$. 

\medskip

Choose $U$ as above, and let $\mathfrak p$ be in $U$. By the Weil conjectures, the
characteristic polynomial of the geometric Frobenius $F_{\mathfrak p}$ has rational
coefficients. By definition of $S$, this implies that it does not have any eigenvalue that is a
root of unity. 

As a consequence, $F_{\mathfrak p}$ acting on the whole cohomology group $H^2(\overline{X},
\Q_{\ell}(1))$ admits $1$ as an eigenvalue of multiplicity $\rho$ and does not have any other
eigenvalue that is a root of unity. It follows that $\rho_{\mathfrak p}\leq \rho$, and finally
that $\rho_{\mathfrak p}= \rho$. This proves the first part of Theorem \ref{jump}.

\bigskip

Now assume that $E$ is a totally real field and that the dimension of $T$ as
a vector space over $E$ is odd. By Proposition \ref{eigenvalues}, every element of $H_{\ell}$
has $1$ as an eigenvalue with multiplicity at least $[E:\Q]$. By definition of the algebraic
monodromy group, if $\mathfrak p$ is a finite place of $k$, then some power of the geometric
Frobenius belongs to $H_{\ell}$. If $X_{\mathfrak p}$ satisfies the Tate conjecture, it follows
that $\rho_{\mathfrak p}\geq \rho + [E:\Q]$.

By Proposition \ref{eigenvalues} again, the set of $h\in H_{\ell}$ such that $h$ admits $1$ as
an eigenvalue of multiplicity $[E:\Q]$ and does not have any other eigenvalue in $S$ is a
dense, Zariski-open subset of $H_{\ell}$.

By Lemma \ref{SO} and Chebotarev's density theorem, we can find a finite extension $k'$ of
$k$ and a set $U$ of finite places $\mathfrak p$ of $k'$ that has density $1$ such
that for any $\mathfrak p\in U$, $X$ has good reduction at $\mathfrak p$ and the geometric
Frobenius $F_{\mathfrak p}$ acting on $T_{\ell}(1)$ admits $1$ as an eigenvalue of multiplicity
$[E:\Q]$ and does not have any other eigenvalue in $S$.

By work of Bogomolov and Zarhin in \cite{BZ09}, the set of finite places where $X$ has good,
ordinary reduction has density $1$ after some finite extension of $k$. As a consequence, we can
assume that $X$ has good, ordinary reduction at every place in $U$.

Choose $U$ as above, and let $\mathfrak p$ be in $U$. By \cite{Ny83}, $X_{\mathfrak p}$
satisfies the Tate conjecture. We can then argue as above to finish the proof of Theorem
\ref{jump}.
\end{proof}

\begin{rk}
 Using Frobenius tori as in \cite{Se81} and the fact that Frobenius tori are maximal tori of
the Mumford-Tate groups for infinitely many primes, one can work directly in the group
of orthogonal similitudes instead of reducing to the special orthogonal group as in Lemma
\ref{SO}.
\end{rk}

\section{Discriminants of N\'eron-Severi groups}

In this section, we discuss properties of the N\'eron-Severi lattices of
specializations of $K3$ surfaces. Once again, we use the notations of Theorem \ref{jump}.

\begin{prop}\label{disc}
 Assume that $E$ is a totally real field and that the dimension of $T$ over $E$ is odd. If
$\mathfrak p$ is a finite place of $k$ such that $X$ has good reduction at $\mathfrak p$,
denote by $\delta(\mathfrak p)\in \Q^*/(\Q^*)^2$ the discriminant of the lattice
$NS(\overline{X_{\mathfrak p}})$ with respect to the intersection product.

There exist infinitely many pairs $(\mathfrak p,\mathfrak q)$ of finite places of
$k$ such that

\begin{enumerate}
 \item $X$ has good, ordinary reduction at both $\mathfrak p$ and $\mathfrak q$.
 \item $\rho_{\mathfrak p}=\rho_{\mathfrak q}=\rho+[E:\Q].$
 \item $\delta(\mathfrak p)\neq \delta(\mathfrak q).$
\end{enumerate}
\end{prop}

\begin{rk}
 A specific case of this result is that the method developed in \cite{vL07} to prove that a
given $K3$ surface over a number field has Picard number $1$ always works in the case $E=\Q$.
We noted in Remark \ref{cex} that it cannot work directly otherwise. 

In the next section, we will adapt the method so as to make it work in every case.
\end{rk}

We start with some easy linear algebra.

\begin{lem}\label{approx}
Let $\ell$ be a prime number, and let $V$ be a free module of finite rank over $\Z_{\ell}$. Let
$g$ be an endomorphism of $V$ such that $g\otimes\Q_{\ell}$ is a semisimple automorphism of
$V\otimes\Q_{\ell}$, and denote by $r$ the multiplicity of $1$ as an eigenvalue of $g$. Let $W$
be the eigenspace associated to the eigenvalue $1$ of $g$. Let $d$ be a positive integer.

Then there exists an integer $N$ with the following property. Let $h$ be an endomorphism of $V$
such that $h\otimes\Q_{\ell}$ is a semisimple automorphism of $V\otimes\Q_{\ell}$. Assume that
$r$ is the multiplicity of $1$ as an eigenvalue of $h$, and let $W'$ be the eigenspace
associated to the eigenvalue $1$ of $h$. If $h$ is congruent to $g$ modulo $\ell^N$, then
$W\otimes\Z/\ell^d\Z =W'\otimes\Z/\ell^d\Z$.
\end{lem}

\begin{rk}
 In particular, if $V$ is endowed with a symmetric bilinear form, the restriction of which to
$W$ is not degenerate, and $N$ is sufficiently large, then the discriminants of $W$ and $W'$
are equal in $\Q_{\ell}^*/(\Q_{\ell}^*)^2$.
\end{rk}

\begin{proof}
 Write $V=W\oplus\widetilde W$, where $\widetilde W$ is a $g$-invariant submodule of $V$. Since
$g\otimes\Q_{\ell}$ does not fix any nonzero element of $\widetilde W\otimes\Q_{\ell}$, there
exists an integer $N$ such that if $g(v)-v\in l^N V$ for some $v\in \widetilde W$, then $v\in
\ell^k\widetilde W$.

Let $h$ be as in the statement of the lemma. By definition of $N$, if $v\in V$ is fixed by $h$,
then $v\otimes \Z/\ell^k\Z\in W\otimes\Z/\ell^k\Z$. With the notation of the lemma, it follows
that $W'\otimes\Z/\ell^k\Z \subset W\otimes\Z/\ell^k\Z$. Since both $W$ and $W'$ are saturated
submodules of $V$ of the same rank $r$, equality follows.
\end{proof}

\begin{proof}[Proof of Proposition \ref{disc}]
First note that the dimension of $T$ as a vector space over $E$ is at least $3$. Indeed, let
$\omega$ be a generator of $T^{2,0}\subset T\otimes\C$, and let $\sigma : E\ra\C$ be the
complex embedding of $E$ satisfying
$$\forall e\in E, e.\omega=\sigma(e)\omega.$$
The complex lines $\C\omega$ and $\C\overline\omega$ are two distinct one-dimensional subspaces
of $T_E\otimes_{\sigma}\C$, where $T_E$ denotes $T$ endowed with the structure of a vector
space over $E$. As a consequence, the dimension of $T$ as a vector space over $E$ is at least
$2$, and at least $3$ since we assumed it to be odd.

\bigskip

Recall that $\psi$ is the bilinear form on $T$ induced by cup-product. As in Proposition
\ref{eigenvalues}, there exists a unique $E$-bilinear form $\phi : T\times T\ra E$ such that
$\psi=Tr_{E/\Q}(\phi)$. Any orthogonal basis of $T_E$ with respect to $\phi$ induces an
orthogonal decomposition of $T$ with respect to $\psi$.
$$T=T_1\oplus \ldots \oplus T_r$$
where the $T_i$ are stable under the action of $E$ and of dimension $1$ as $E$-vector spaces.

By the same reasoning as above, since the $T_i$ are one-dimensional over $E$, there is no
integer $i$ such that $T_i\otimes\C$ contains the two-dimensional space $T^{2,0}\oplus
T^{0,2}$. 

The signature of $\psi$ on $T$ is $(2, dim(T)-2)$. By the Hodge index theorem and the remark above, the signature of the restriction of $\psi$ to $T_i$ is either $(0, [E:\Q])$ or $(1, [E:\Q]-1)$. Since the dimension of $T$ over $E$ is at least $3$, both these signatures appear, and this implies that there exist integers $i$ and
$j$ such that the discriminant of $T_i$ is negative and the discriminant of $T_j$ is positive. Let $\delta_i$ and $\delta_j$ be these two discriminants in $\Q^*/(\Q^*)^2$.

\bigskip

Since $\delta_i\neq \delta_j$ in $\Q^*/(\Q^*)^2$, there exists a prime number such that the
images of $\delta_i$ and $\delta_j$ in $\Q_{\ell}^*/(\Q_{\ell}^*)^2$ are different. If $W$ is
any subspace of $T_{\ell}$ such that the restriction of $\psi_{\ell}$ to $W$ is non-degenerate,
let $\delta(W)$ denote the discriminant of $W$ in $\Q_{\ell}^*/(\Q_{\ell}^*)^2$

By Lemma \ref{approx}, Proposition \ref{eigenvalues} and Chebotarev's density theorem, we can
find, for any positive integer $d$, infinitely many pairs $(\mathfrak p,\mathfrak q)$ of finite
places of $k$ such that
\begin{enumerate}
 \item $X$ has good, ordinary reduction at both $\mathfrak p$ and $\mathfrak q$.
 \item $\rho_{\mathfrak p}=\rho_{\mathfrak q}=\rho+[E:\Q].$
 \item If $F_{\mathfrak p}$ (resp. $F_{\mathfrak q}$) denotes the geometric Frobenius at
$\mathfrak p$ (resp. $\mathfrak q$) acting on $T_{\ell}(1)$, and $W_{\mathfrak p}$ (resp.
$W_{\mathfrak q}$) denotes the eigenspace associated to the eigenvalue $1$ of $F_{\mathfrak p}$
(resp. $F_{\mathfrak q}$), then $\delta(W_{\mathfrak p})=\delta_i$ in
$\Q_{\ell}^*/(\Q_{\ell}^*)^2$ (resp. $\delta(W_{\mathfrak q})=\delta_j$ in
$\Q_{\ell}^*/(\Q_{\ell}^*)^2$).
 \item The geometric Frobenius $F_{\mathfrak p}$ (resp. $F_{\mathfrak q}$) denotes the
geometric Frobenius at $\mathfrak p$ (resp. $\mathfrak q$) acting on $T_{\ell}(1)$ does not
have any eigenvalue different from $1$ that is a root of unity.
\end{enumerate}
Proposition \ref{disc} immediately follows by the Tate conjecture for ordinary $K3$ surfaces.
\end{proof}

\begin{rk}
 The proof above shows that the density of pairs $(\mathfrak p, \mathfrak q)$ as in the
proposition is positive.
\end{rk}

\section{Computing the Picard number over number fields}

This section is devoted to a proof of Theorem \ref{compute}. Given a projective $K3$ surface
over a number field $k$, we want to compute the Picard number of $X$ using the equations of
$X$ in a projective embedding.

There are two steps in our approach. The first one is finding sufficiently many divisors on
$X$, and the second is proving that these divisors generate the N\'eron-Severi group of $X$ --
at least rationally.

In case we want to prove that the $K3$ surface has Picard number $1$, the first step is
vacuous, as we already have a divisor given by a hyperplane section. In general, the first
step is done by going through the Hilbert schemes of curves in the projective space we are
working in and doing elimination theory to find curves on $X$. After a finite number of
computations, this will allow us to find divisors on $X$ that span the N\'eron-Severi group.

The second step will be done by reducing to finite characteristic and using our results above.

\bigskip

However, this is not sufficient. Indeed, the field $E$ of endomorphisms of the transcendental
part of the Hodge structure of $X$ plays a role in the behavior of the Picard number after
specialization, and in case $E$ is a totally real field strictly containing $\Q$ such that $T$
is of odd dimension over $E$, this leads to some loss of accuracy in the estimates reduction
at finite places can provide.

This problem will be solved by studying codimension $2$ varieties in $X\times X$. Assuming the
Hodge conjecture for $X\times X$, these determine the field $E$, which will allow us to
conclude.

We start by the following result.

\begin{prop}\label{check}
 Let $X$ be a $K3$ surface over a number field $k$. Assume we are given the equations of $X$
in some projective embedding. 

Let $T$ be the transcendental part of $H^2(X, \Q)$, and let $E$ be the field of endomorphisms
of the Hodge structure $T$. 

Assume that we know that the Picard number of $X$ is greater or equal to some integer $\rho$,
and that the degree of $E$ over $\Q$ is greater or equal to some integer $d$.

Then there exists an algorithm with the following properties: 
\begin{enumerate}
 \item  Suppose that the Picard number of $X$ is actually $\rho$. Then the algorithm terminates
 unless $E$ is totally real, the dimension of $T$ as a vector space over $E$ is odd and
$d<[E:\Q]$.
 \item If the algorithm terminates, it proves that Picard number of $X$ is $\rho$.
\end{enumerate}
\end{prop}

\begin{proof}
 Let $\rho'$ be the actual Picard number of $X$. We know that $\rho'\geq\rho$. Using the Weil
conjectures \cite{De74}, we can compute the characteristic polynomial of Frobenius at any finite place of
$k$, see \cite{vL07}, \cite{EJ11}. This allows in particular to check whether $X$ has good,
ordinary reduction at a given place $\mathfrak p$, and to compute the numbers $\rho_{\mathfrak
p}$ for such places. Using the Artin-Tate formula, one can also compute the discriminants
$\delta(\mathfrak p)$ as in Proposition \ref{disc}.

We start computing $\rho_{\mathfrak p}$ and $\delta(\mathfrak p)$ for all ordinary places
$\mathfrak p$. 

Let us distinguish three cases. First assume that $E$ is a CM field. By Theorem \ref{jump}, we
can find $\mathfrak p$ with $\rho_{\mathfrak p}=\rho'$. If it happens that $\rho$, the lower
bound for the Picard number of $X$ that we were given, is equal to the actual Picard number
$\rho'$ (that we do not know yet), the computation at $\mathfrak p$ together with this lower
bound proves that $X$ has Picard number $\rho=\rho'$.

\bigskip
Now assume that $E$ is totally real and the dimension of $T$ as a vector
space over $E$ is even. In that case, Theorem \ref{jump} allows us to make the
same conclusion.

\bigskip

The last case happens when $E$ is totally real and the dimension of $T$ as a vector
space over $E$ is odd. By Proposition \ref{disc}, the finite field computations give us two
finite places $\mathfrak p$ and $\mathfrak q$ of $k$ where $X$ has good, ordinary reduction,
with $\rho_{\mathfrak{p}}=\rho_{\mathfrak q}=\rho'+[E:\Q]$, and $\delta(\mathfrak p)\neq
\delta(\mathfrak q)$. 

Since $\delta(\mathfrak p)\neq \delta(\mathfrak q)$, we know that the specialization maps
$NS(\overline X)\ra NS(\overline X_{\mathfrak p})$ and $NS(\overline X)\ra NS(\overline
X_{\mathfrak p})$ are not surjective. This means that $NS(X_{\mathfrak p})\cap T_l(1)$ is
nonzero in $H^2(\overline X_{\mathfrak p}, \Q_l(1))$.

Now we know by the analysis in the proof of Theorem \ref{jump} that this intersection is
stable under the action of $E$. As a consequence, its dimension is at least $[E:\Q]\geq d$.
This gives us the estimation 
$$\rho'\leq \rho_{\mathfrak p}-d.$$
In case $d$ happens to be equal to the actual degree $[E:\Q]$ and $\rho=\rho'$, these
estimates allow is to prove that $X$ has Picard number $\rho=\rho'$.
\end{proof}

\begin{rk}
 In case $\rho=\rho'=1$ and $E=\Q$, this proves the method of \cite{vL07} always works.
\end{rk}

\begin{proof}[Proof of Theorem \ref{compute}]
 Let $X$, $E$ and $T$ be as above. Let $\rho'$ be the Picard number of $X$ and $d'$ be the
degree of $E$ over $\Q$. By Proposition \ref{check}, we only need to be able to prove that the
Picard number of $X$ is at least $\rho'$ and the degree of $E$ over $\Q$ is at least $d'$.

The assertion on the Picard number is theoretically -- although not computationally -- easy.
One
can go through Hilbert schemes of curves in the projective space where $X$ is given
and check, using elimination theory, for curves that happen to lie on $X$. Computing
intersection matrices with this divisors on $X$, one can find divisors that span a
$\rho'$-dimensional subset of the N\'eron-Severi group of $X$.

Running these Hilbert schemes computations alongside the computations of Proposition
\ref{check} allows for a computation of the Picard number of $X$ unless $E$ is a totally real
field strictly containing $\Q$ such that $T$ is of odd dimension over $E$.

\bigskip

To deal with the latter case, one has to work on $X\times X$. If one assumes the Hodge
conjecture for $X\times X$, then elements of $E$ are induced by codimension $2$ cycles in
$X\times X$. As above, one can use Hilbert schemes to find codimension $2$ subschemes in
$X\times X$. 

Given such a subscheme $Z$, the action of $Z$ on $T$ can be determined by first
computing the characteristic polynomial of the correspondence $H^2(X, \Q)\ra H^2(X, \Q)$ by
computing intersection numbers between $T$ and the various subschemes obtained by composing
the correspondence induced by $Z$ with itself. 

Factoring the characteristic polynomial, this gives candidates for the algebraic number
$\lambda$ such that $[Z]_*\eta=\lambda\eta$, where $\eta$ is a nonzero algebraic $2$-form on
$X$. An approximate computation can then determine $\lambda$. The degree of $\lambda$ over
$\Q$ is a lower bound for $[E:\Q]$. 

By the primitive element theorem, it is easy to see that one can find $Z$ such that this
computation gives an optimal estimate for the degree of $E$. Using Proposition \ref{check},
this concludes the proof. 

\bigskip

In conclusion, an algorithm to compute Picard number of $K3$ surfaces works as follows. Let $X$
be a $K3$ surface. Run the three following algorithms alongside each other.
\begin{enumerate}
 \item Going through Hilbert schemes of a suitable projective space, find divisors on $X$ and
compute the dimension of their span in the N\'eron-Severi group via intersection theory. This
gives a lower bound for the Picard number.
 \item Going through Hilbert schemes of a suitable projective space, find codimension $2$
cycles in $X\times X$. Using intersection theory again, use these to get a lower bound on the
field $E$ of endomorphisms of the transcendental part of $H^2(X, \Q)$.
 \item Going through finite places $\mathfrak p$ of $k$, compute the Picard number and the
discriminant of the N\'eron-Severi group of $\overline X_{\mathfrak p}$ by counting points over
finite fields. Using the preceding step, get an upper bound on the Picard number of $X$.
\end{enumerate}
We showed that the estimates provided by the method solve the problem unconditionally unless
$E$ is a totally real field strictly containing $\Q$ and the transcendental part of $H^2(X,
\Q)$ is of odd dimension over $E$. In the latter case, the estimates above are sufficiently
precise to compute the Picard number if we assume the Hodge conjecture for $X\times X$.

\end{proof}

\begin{rk}
 It seems that the computations of the second step above would be very
lengthy to do in practice. We however wanted to point out that they could be done
theoretically. 

Note that the computations terminate much faster in most cases, since $E=\Q$ for the majority of $K3$
complex surfaces, in the sense of Baire category.
\end{rk}

\bibliography{bibK3}{}
\bibliographystyle{plain}
\end{document}